\documentclass[a4, 12pt]{amsart}
\usepackage{amssymb}
\usepackage{amstext}
\usepackage{amsmath}
\usepackage{amscd}
\usepackage{amsthm}
\usepackage{amsfonts}
\usepackage{enumerate}
\usepackage{graphicx}
\usepackage{latexsym}
\usepackage[all]{xy}

\theoremstyle{plain}
\newtheorem{thm}{Theorem}[section]
\newtheorem*{thm*}{Theorem}
\newtheorem*{cor*}{Corollary}
\newtheorem*{prop*}{Proposition}

\newtheorem*{mthm}{Main Theorem}

\newtheorem{prop}[thm]{Proposition}
\newtheorem{lem}[thm]{Lemma}

\newtheorem{cor}[thm]{Corollary}

\newtheorem*{claim*}{Claim}

\theoremstyle{definition}
\newtheorem{defn}[thm]{Definition}
\newtheorem{rem}[thm]{Remark}

\newtheorem*{conj*}{Conjecture}

\newtheorem*{mpf}{Proof of Main Theorem}

\theoremstyle{remark}

\numberwithin{equation}{thm}
\def\Hom{\operatorname{Hom}}

\def\Mod{\operatorname{Mod}}

\def\p{\mathfrak p}
\def\q{\mathfrak q}

\def\Z{\Bbb Z}

\def\supp{\operatorname{supp}}

\def\Ass{\operatorname{Ass}}

\def\Spec{\operatorname{Spec}}

\def\D{{\mathcal D}}

\def\X{{\mathcal X}}

\def\M{{\mathcal M}}

\def\AA{\left\{
\begin{matrix}
\text{localizing}\\
\text{subcategories}\\
\text{of }\D(A)
\end{matrix}
\right\}}

\def\AB{\left\{
\begin{matrix}
\text{subsets}\\
\text{of }\Spec A
\end{matrix}
\right\}}

\def\AC{\left\{
\begin{matrix}
\text{E-stable subcategories}\\
\text{of }\Mod A\text{ closed}\\
\text{under direct sums}\\
\text{and summands}
\end{matrix}
\right\}}

\def\BA{\left\{
\begin{matrix}
\text{localizing}\\
\text{subcategories}\\
\text{of }\D(A)\text{ closed}\\
\text{under homology}
\end{matrix}
\right\}}

\def\BB{\left\{
\begin{matrix}
\text{coherent}\\
\text{subsets}\\
\text{of }\Spec A
\end{matrix}
\right\}}

\def\BC{\left\{
\begin{matrix}
\text{thick subcategories}\\
\text{of }\Mod A\text{ closed}\\
\text{under direct sums}
\end{matrix}
\right\}}

\def\CA{\left\{
\begin{matrix}
\text{smashing}\\
\text{subcategories}\\
\text{of }\D(A)
\end{matrix}
\right\}}

\def\CB{\left\{
\begin{matrix}
\text{subsets}\\
\text{of }\Spec A\\
\text{closed under}\\
\text{specialization}
\end{matrix}
\right\}}

\def\CC{\left\{
\begin{matrix}
\text{localizing}\\
\text{subcategories}\\
\text{of }\Mod(A)
\end{matrix}
\right\}}

\begin{document}

\setlength{\baselineskip}{18pt}

\title{On localizing subcategories of derived categories}
\author{Ryo Takahashi}
\address{Department of Mathematical Sciences, Faculty of Science, Shinshu University, 3-1-1 Asahi, Matsumoto, Nagano 390-8621, Japan}
\email{takahasi@math.shinshu-u.ac.jp}
\begin{abstract}
Let $A$ be a commutative noetherian ring.
In this paper, we interpret localizing subcategories of the derived category of $A$ by using subsets of $\Spec A$ and subcategories of the category of $A$-modules.
We unify theorems of Gabriel, Neeman and Krause.
\end{abstract}
\maketitle
\section{Introduction}

Let $A$ be a commutative noetherian ring.
In this paper, we investigate the relationship among subcategories of the derived category $\D(A)$ of $A$, subcategories of the category $\Mod A$ of $A$-modules, and subsets of the prime spectrum $\Spec A$ of $A$ (i.e. $\Spec A$ is the set of prime ideals of $A$).

In the early 1960s, Gabriel \cite{Gabriel} showed the following.

\begin{thm}[Gabriel]
There is an inclusion-preserving bijection between the set of localizing subcategories of $\Mod A$ and the set of subsets of $\Spec A$ closed under specialization.
\end{thm}

Thirty years later, Neeman \cite{Neeman} proved the following result, which generalizes a theorem of Hopkins \cite{Hopkins}.

\begin{thm}[Neeman]
The assignment $\X\mapsto\supp\X$ makes an inclusion-preserving bijection from the set of localizing subcategories of $\D(A)$ to the set of subsets of $\Spec A$, which induces an inclusion-preserving bijection from the set of smashing subcategories of $\D(A)$ to the set of subsets of $\Spec A$ closed under specialization.
The inverse map sends a subset $\Phi$ of $\Spec A$ to the localizing subcategory of $\D(A)$ generated by $\{ k(\p)\}_{\p\in\Phi}$.
\end{thm}

Here, $\supp\X$ denotes the set of prime ideals $\p$ of $A$ such that $\p\in\supp X$ for some $X\in\X$, where $\supp X$ denotes the set of prime ideals $\p$ such that $k(\p)\otimes_A^{\bf L}X\ne 0$ in $\D(A)$ ($k(\p)$ denotes the residue field $A_\p/\p A_\p$).

Recently, Krause \cite{Krause} generalized the above Gabriel's result, and corrected a theorem of Hovey \cite{Hovey}.

\begin{thm}[Krause]
The assignment $\M\mapsto\supp\M$ makes an inclusion-preserving bijection from the set of thick subcategories of $\Mod A$ closed under direct sums and the set of coherent subsets of $\Spec A$.
The inverse map is given by $\Phi\mapsto(\supp^{-1}\Phi)_0$.
\end{thm}

Here, $\supp^{-1}\Phi$ denotes the full subcategory of $\D(A)$ consisting of all complexes $X$ such that $\supp X$ is contained in $\Phi$, and for a subcategory $\X$ of $\D(A)$, $\X_0$ denotes the full subcategory of $\Mod A$ consisting of all modules whose corresponding complexes are in $\X$.

Let $E(M)=(0\to E^0(M)\to E^1(M)\to E^2(M)\to\cdots)$ denote the minimal injective resolution of an $A$-module $M$.
We say that a full subcategory $\M$ of $\Mod A$ is E-stable provided that a module $M$ is in $\M$ if and only if so is $E^i(M)$ for all $i\ge 0$.
We denote by $\widetilde\M$ the localizing subcategory of $\D(A)$ generated by $\M$, and by $\overline\M$ the localizing subcategory of $\D(A)$ consisting of all complexes each of whose homology modules is in $\M$.
A subcategory $\X$ of $\D(A)$ is said to be closed under homology if (the corresponding complex of) any homology module of any complex in $\X$ is also in $\X$.
Our main result is the following, which contains all of the above three theorems.

\begin{mthm}
One has the following commutative diagram of inclusion-preserving bijections.
\[\xymatrix{
{\AA} \ar@<2mm>[rr]^{(-)_0} \ar@{}[rr]|\cong \ar@<2mm>[dr]^{\supp} & & {\AC} \ar@<2mm>[ll]^{\widetilde{(-)}} \ar@{}[dl]|\cong  \ar@<-2mm>[dl]_\supp \\
& {\AB} \ar@{}[u] \ar@<2mm>[lu]^{\supp^{-1}} \ar@{}[lu]|\cong \ar@<-2mm>[ru]_{(\supp^{-1}(-))_0} &
}\]

Moreover, restricting this diagram, one has the following two commutative diagrams of inclusion-preserving bijections.
\[\xymatrix{
{\BA} \ar@<2mm>[rr]^{(-)_0} \ar@{}[rr]|\cong \ar@<2mm>[dr]^{\supp} & &  {\BC} \ar@<2mm>[ll]^{\overline{(-)}} \ar@{}[dl]|\cong  \ar@<-2mm>[dl]_\supp \\
& {\BB} \ar@{}[u] \ar@<2mm>[lu]^{\supp^{-1}} \ar@{}[lu]|\cong \ar@<-2mm>[ru]_{(\supp^{-1}(-))_0} &
}\]
\[\xymatrix{
{\CA} \ar@<2mm>[rr]^{(-)_0} \ar@{}[rr]|\cong \ar@<2mm>[dr]^{\supp} & & {\CC} \ar@<2mm>[ll]^{\overline{(-)}} \ar@{}[dl]|\cong  \ar@<-2mm>[dl]_\supp \\
& {\CB} \ar@{}[u] \ar@<2mm>[lu]^{\supp^{-1}} \ar@{}[lu]|\cong \ar@<-2mm>[ru]_{(\supp^{-1}(-))_0} &
}\]

Thus, one obtains the following commutative diagram.
$$
\begin{CD}
\AA @. \cong @. \AB @. \cong @. \AC \\
\cup @.  @. \cup @. @. \cup \\
\BA @. \cong @. \BB @. \cong @. \BC \\
\cup @. @. \cup @. @. \cup \\
\CA @. \cong @. \CB @. \cong @. \CC
\end{CD}
$$
\end{mthm}

There are some related works other than ones cited above; see \cite{Balmer, BIK, GP, GP2, wide, Thomason} for example.
In the next section, we will prove this Main Theorem after stating precise definitions and showing preliminary results.


\section{Proof of Main Theorem}

Throughout this section, let $A$ be a commutative noetherian ring.
By a {\em subcategory}, we always mean a full subcategory which is closed under isomorphisms.
We denote the category of $A$-modules by $\Mod A$ and the derived category of $\Mod A$ by $\D(A)$.
For an $A$-module $M$, let
$$
C_M=(\cdots \to 0 \to M \to 0 \to \cdots)
$$
be the complex with $M$ in degree zero.
We will often identify $M$ with $C_M$.

First of all, we recall the definitions of a triangulated subcategory and a localizing subcategory of $\D(A)$.
For a (cochain) $A$-complex $X$ and an integer $n$, we denote by $X[n]$ the complex $X$ shifted by $n$ degrees; its module in degree $i$ is $X^{n+i}$ for each integer $i$.

\begin{defn}
Let $\X$ be a subcategory of $\D(A)$.
\begin{enumerate}[(1)]
\item
We say that $\X$ is {\em triangulated} provided that for every exact triangle $X\to Y\to Z\to X[1]$ in $\D(A)$, if two of $X$, $Y$ and $Z$ are in $\X$, then so is the third.
\item
We say that $\X$ is {\em localizing} if $\X$ is triangulated and closed under (arbitrary) direct sums.
\end{enumerate}
\end{defn}

\begin{rem}
\begin{enumerate}[(1)]
\item
Triangulated subcategories of $\D(A)$ are {\em closed under shifts}: if a complex $X$ is in a triangulated subcategory $\X$ of $\D(A)$, then $X[n]$ is also in $\X$ for every integer $n$.

In fact, it follows from the triangle $X\overset{=}{\to}X\to 0\to X[1]$ that $0$ is in $\X$, and it follows from the triangles $X\to 0\to X[1]\overset{=}{\to} X[1]$ and $X[-1]\to 0\to X\overset{=}{\to} X$ that $X[1],X[-1]$ are in $\X$.
An inductive argument shows that $X[n]$ is in $\X$ for every $n\in\Z$.
\item
Localizing subcategories of $\D(A)$ are closed under direct summands; see \cite[Proposition 1.6.8]{Neemanbook}.
\end{enumerate}
\end{rem}

The support of a complex is defined as follows.

\begin{defn}
The (small) {\em support} $\supp X$ of an $A$-complex $X$ is defined as the set of prime ideals $\p$ of $A$ satisfying $k(\p)\otimes_A^{\bf L}X\ne 0$ in $\D(A)$, where $k(\p)$ denotes the residue field $A_\p/\p A_\p$ of the local ring $A_\p$.
\end{defn}

Here we state basic properties of support.

\begin{lem}\label{supp}
\begin{enumerate}[\rm (1)]
\item
Let $X\to Y\to Z\to X[1]$ be an exact triangle in $\D(A)$.
Then one has the following inclusion relations:
\begin{align*}
\supp X & \subseteq \supp Y\cup\supp Z, \\
\supp Y & \subseteq \supp Z\cup\supp X, \\
\supp Z & \subseteq \supp X\cup\supp Y.
\end{align*}
\item
The equality
$$
\supp\bigg(\bigoplus_{\lambda\in\Lambda}X_\lambda\bigg)=\bigcup_{\lambda\in\Lambda}\supp X_\lambda
$$
holds for any family $\{ X_\lambda\}_{\lambda\in\Lambda}$ of $A$-complexes.
\item
Let $s$ be an integer, and let $X=(\cdots\to X^{s-2}\to X^{s-1}\to X^s\to 0)$ be an $A$-complex.
Then
$$
\supp X\subseteq\bigcup_{i\le s}\supp X^i.
$$
\end{enumerate}
\end{lem}

\begin{proof}
(1) Let $\p$ be a prime ideal in $\supp X$.
Then $k(\p)\otimes_A^{\bf L}X$ is nonzero.
There is an exact triangle
$$
k(\p)\otimes_A^{\bf L}X\to k(\p)\otimes_A^{\bf L}Y\to k(\p)\otimes_A^{\bf L}Z \to k(\p)\otimes_A^{\bf L}X[1],
$$
which says that either $k(\p)\otimes_A^{\bf L}Y$ or $k(\p)\otimes_A^{\bf L}Z$ is nonzero.
Thus $\p$ is in the union of $\supp Y$ and $\supp Z$.
The other inclusion relations are similarly obtained.

(2) One has $k(\p)\otimes_A^{\bf L}(\bigoplus_{\lambda\in\Lambda}X_\lambda)\cong\bigoplus_{\lambda\in\Lambda}(k(\p)\otimes_A^{\bf L}X_\lambda)$ for $\p\in\Spec A$.
Hence $k(\p)\otimes_A^{\bf L}(\bigoplus_{\lambda\in\Lambda}X_\lambda)$ is nonzero if and only if $k(\p)\otimes_A^{\bf L}X_\lambda$ is nonzero for some $\lambda\in\Lambda$.

(3) Assume that a prime ideal $\p$ of $A$ satisfies $k(\p)\otimes_A^{\bf L}X^i=0$ for every $i\le s$.
Let $F=(\cdots\to F^{-2}\to F^{-1}\to F^0\to 0)$ be a free resolution of the $A$-module $k(\p)$.
Then the complex $F\otimes_AX^i=(\cdots\to F^{-2}\otimes_AX^i\to F^{-1}\otimes_AX^i\to F^0\otimes_AX^i\to 0)$ is exact for every $i\le s$.
We have a commutative diagram
$$
\begin{CD}
@. \vdots @. \vdots @. \vdots \\
@. @VVV @VVV @VVV \\
(\cdots @>>> F^{-2}\otimes_AX^{s-2} @>>> F^{-2}\otimes_AX^{s-1} @>>> F^{-2}\otimes_AX^s @>>> 0) \\
@. @VVV @VVV @VVV \\
(\cdots @>>> F^{-1}\otimes_AX^{s-2} @>>> F^{-1}\otimes_AX^{s-1} @>>> F^{-1}\otimes_AX^s @>>> 0) \\
@. @VVV @VVV @VVV \\
(\cdots @>>> F^0\otimes_AX^{s-2} @>>> F^0\otimes_AX^{s-1} @>>> F^0\otimes_AX^s @>>> 0) \\
@. @VVV @VVV @VVV \\
@. 0 @. 0 @. 0
\end{CD}
$$
with exact columns.
Considering the spectral sequence of the double complex $F\otimes_AX$, we see that the total complex of $F\otimes_AX$ is exact.
This means that $k(\p)\otimes_A^{\bf L}X=0$.
\end{proof}

We say that a subcategory $\X$ of $\D(A)$ is {\em closed under left complexes} provided that for any $A$-complex $X=(\cdots\to X^{s-2}\to X^{s-1}\to X^s\to 0)$ bounded above, if each $X^i$ is in $\X$, then $X$ is also in $\X$.
For a subset $\Phi$ of $\Spec A$, we denote by $\supp^{-1}\Phi$ the subcategory of $\D(A)$ consisting of all $A$-complexes $X$ with $\supp X\subseteq\Phi$.
The following proposition immediately follows from Lemma \ref{supp}.

\begin{prop}\label{-1}
Let $\Phi$ be a subset of $\Spec A$.
Then $\supp^{-1}\Phi$ is a localizing subcategory of $\D(A)$ closed under left complexes.
\end{prop}

We denote the set of associated primes of an $A$-module $M$ by $\Ass M$, and the injective hull of $M$ by $E(M)$.

\begin{lem}\label{ass}
\begin{enumerate}[\rm (1)]
\item
For an $A$-module $M$ we have a direct sum decomposition
$$
E(M)\cong\bigoplus_{\p\in\Ass M}E(A/\p)^{\oplus \Lambda_\p},
$$
where $\Lambda_\p$ is a nonempty set.
\item
The equality
$$
\supp I=\Ass I
$$
holds for every injective $A$-module $I$.
\end{enumerate}
\end{lem}

\begin{proof}
(1) This assertion can be shown by using \cite[Theorem 3.2.8]{BH}.

(2) Let $\p,\q$ be prime ideals of $A$.
Then we easily see that there are isomorphisms
$$
k(\p)\otimes_A^{\bf L}(E(A/\q)_\p)\cong k(\p)\otimes_A^{\bf L}E(A/\q)\cong(k(\p)_\q)\otimes_A^{\bf L}E(A/\q).
$$
Therefore the complex $k(\p)\otimes_A^{\bf L}E(A/\q)$ is nonzero if and only if $\p=\q$.
The assertion follows from this fact and (1).
\end{proof}

We now recall the definition of a smashing subcategory.

\begin{defn}
Let $\X$ be a localizing subcategory of $\D(A)$.
\begin{enumerate}[(1)]
\item
An object $C\in\D(A)$ is called {\em $\X$-local} if $\Hom_{\D(A)}(X,C)=0$ for any $X\in\X$.
\item
A morphism $f:C\to L$ is called a {\em localization} of $C$ by $\X$ if $L$ is $\X$-local, and $\Hom_{\D(A)}(f,L'):\Hom_{\D(A)}(L,L')\to\Hom_{\D(A)}(C,L')$ is an isomorphism for any $\X$-local object $L'\in\D(A)$.
\item
$\X$ is called {\em smashing} if localization by $\X$ commutes with direct sums.
\end{enumerate}
\end{defn}

For a subcategory $\X$ of $\D(A)$, we denote by $\supp\X$ the set of prime ideals $\p$ of $A$ such that $\p\in\supp X$ for some $X\in\X$.
We describe a theorem of Neeman \cite{Neeman} in the following form.

\begin{thm}\label{neeman}
\begin{enumerate}[\rm (1)]
\item
One has maps
$$
\AA
\begin{smallmatrix}
f\\
\longrightarrow\\
\longleftarrow\\
g
\end{smallmatrix}
\AB
$$
defined by $f(\X)=\supp\X$ and $g(\Phi)=\supp^{-1}\Phi$.
The map $f$ is an inclusion-preserving bijection and $g$ is its inverse map.
\item
One has maps
$$
\CA
\begin{smallmatrix}
f\\
\longrightarrow\\
\longleftarrow\\
g
\end{smallmatrix}
\CB
$$
defined by $f(\X)=\supp\X$ and $g(\Phi)=\supp^{-1}\Phi$.
The map $f$ is an inclusion-preserving bijection and $g$ is its inverse map.
\end{enumerate}
\end{thm}

\begin{proof}
(1) Proposition \ref{-1} guarantees that $g$ is well-defined.
Let $\Phi$ be a subset of $\Spec A$.
Then the inclusion $\supp(\supp^{-1}\Phi)\subseteq\Phi$ clearly holds.
It follows from Lemma \ref{ass}(2) that $\supp E(A/\p)=\Ass E(A/\p)=\{\p\}\subseteq\Phi$ for every $\p\in\Phi$, which yields the opposite inclusion $\supp(\supp^{-1}\Phi)\supseteq\Phi$.
Therefore we have the equality $\supp(\supp^{-1}\Phi)=\Phi$, which shows that $fg$ is the identity map.
Since $f$ is a bijective map by virtue of \cite[Theorem 2.8]{Neeman}, $g$ is the inverse map of $f$.
It is easy to check that $f$ is inclusion-preserving.

(2) This follows from \cite[Theorem 3.3]{Neeman} and (1).
\end{proof}

Combining Theorem \ref{neeman}(1) with \cite[Theorem 2.8]{Neeman} and Proposition \ref{-1}, we obtain the following result.

\begin{cor}\label{gene}
\begin{enumerate}[\rm (1)]
\item
For every subset $\Phi$ of $\Spec A$, $\supp^{-1}\Phi$ is the localizing subcategory of $\D(A)$ generated by $\{ k(\p)\}_{\p\in\Phi}$.
\item
Any localizing subcategory of $\D(A)$ is closed under left complexes.
\end{enumerate}
\end{cor}

Krause \cite{Krause} introduces the notion of a coherent subset of $\Spec A$:

\begin{defn}
A subset $\Phi$ of $\Spec A$ is called {\em coherent} if every homomorphism $f:I^0\to I^1$ of injective $A$-modules with $\Ass I^i\subseteq\Phi$ for $i=1,2$ can be completed to an exact sequence $I^0\overset{f}{\to}I^1\to I^2$ of injective $A$-modules with $\Ass I^2\subseteq\Phi$.
\end{defn}

To relate coherent subsets of $\Spec A$ to localizing subcategories of $\D(A)$, we make the following definition.

\begin{defn}
Let $\X$ be a subcategory of $\D(A)$.
\begin{enumerate}[(1)]
\item
We say that $\X$ is {\em closed under homology} if $H^i(X)$ is in $\X$ for all $X\in\X$ and $i\in\Z$.
\item
We say that $\X$ is {\em H-stable} provided that a complex $X$ is in $\X$ if and only if so is $H^i(X)$ for every $i\in\Z$.
\end{enumerate}
\end{defn}

We have the following one-to-one correspondence.

\begin{thm}\label{mt}
One has maps
$$
\BA
\begin{smallmatrix}
f\\
\longrightarrow\\
\longleftarrow\\
g
\end{smallmatrix}
\BB
$$
defined by $f(\X)=\supp\X$ and $g(\Phi)=\supp^{-1}\Phi$.
The map $f$ is an inclusion-preserving bijection and $g$ is its inverse map.
\end{thm}

\begin{proof}
Let $\X$ be a localizing subcategory of $\D(A)$ closed under homology.
Then we have $\X=\supp^{-1}(\supp\X)$ by Theorem \ref{neeman}(1).
It is seen from \cite[Theorem 5.2]{Krause} that $\supp\X$ is coherent.
Hence $f$ is well-defined.
From Proposition \ref{-1} and \cite[Theorem 5.2]{Krause} we see that $g$ is well-defined.
Theorem \ref{neeman}(1) shows that $f$ is an inclusion-preserving bijective map and $g$ is the inverse map of $f$.
\end{proof}

By definition, an H-stable subcategory of $\D(A)$ is closed under homology.
The converse of this statement also holds:

\begin{cor}\label{lchst}
\begin{enumerate}[\rm (1)]
\item
Any localizing subcategory of $\D(A)$ closed under homology is H-stable.
\item
Any smashing subcategory of $\D(A)$ is H-stable.
\end{enumerate}
\end{cor}

\begin{proof}
(1) Let $\X$ be a localizing subcategory of $\D(A)$ closed under homology.
Then by Theorem \ref{mt} we have $\X=\supp^{-1}\Phi$ for some coherent subset $\Phi$ of $\Spec A$.
It follows from \cite[Theorem 5.2]{Krause} that $\X$ is H-stable.

(2) Let $\X$ be a smashing subcategory of $\D(A)$.
Then $\Phi:=\supp\X$ is closed under specialization by Theorem \ref{neeman}(2).
It follows from \cite[Proposition 4.1(2)]{Krause} that $\Phi$ is coherent.
Theorem \ref{neeman}(1) yields $\X=\supp^{-1}\Phi$, which is H-stable by Theorem \ref{mt} and (1).
\end{proof}

Following \cite{Krause}, we define a thick subcategory of modules as follows.

\begin{defn}
A subcategory $\M$ of $\Mod A$ is called {\em thick} provided that for any exact sequence
$$
M_1\to M_2\to M_3\to M_4\to M_5
$$
of $A$-modules, if $M_i$ is in $\M$ for $i=1,2,4,5$, then so is $M_3$.
\end{defn}

\begin{rem}
\begin{enumerate}[(1)]
\item
A subcategory of $\Mod A$ is thick if and only if it is closed under kernels, cokernels and extensions.
\item
If a subcategory of $\Mod A$ is closed under kernels or cokernels, then it is closed under direct summands.
In particular, every thick subcategory of $\Mod A$ is closed under direct summands, and contains the zero module $0$.

Indeed, assume that the direct sum $M=N\oplus L$ of two $A$-modules $N,L$ is in a subcategory $\M$ of $\Mod A$.
Then the exact sequence
$$
0\to N\to M\overset{\left(\begin{smallmatrix}
0 & 0 \\
0 & 1
\end{smallmatrix}\right)}{\longrightarrow} M\to N\to 0
$$
of $A$-modules shows that $N$ is in $\M$ if $\M$ is closed under kernels or cokernels.
\end{enumerate}
\end{rem}

For an $A$-module $M$, let
$$
E(M)=(0\to E^0(M)\to E^1(M)\to E^2(M)\to\cdots)
$$
denote the minimal injective resolution of $M$.
(Recall that a minimal injective resolution of a given $A$-module is uniquely determined up to isomorphism; see \cite[Page 99]{BH}.)

\begin{defn}
We say that a subcategory $\M$ of $\Mod A$ is {\em E-stable} provided that a module $M$ is in $\M$ if and only if so is $E^i(M)$ for every $i\ge 0$.
\end{defn}

\begin{prop}\label{thest}
Every thick subcategory of $\Mod A$ closed under direct sums is E-stable.
\end{prop}

\begin{proof}
Let $\M$ be a thick subcategory of $\Mod A$ closed under direct sums.
Then $\M$ is closed under cokernels and injective hulls by \cite[Lemma 3.5]{Krause}.
Hence $E^i(M)$ is in $\M$ for every $M\in\M$ and $i\ge 0$.
Conversely, let $M$ be an $A$-module with $E^i(M)\in\M$ for any $i\ge 0$.
There is an exact sequence
$$
0 \to M \to E^0(M) \to E^1(M)
$$
of $A$-modules, and $M$ is in $\M$ by the closedness of $\M$ under kernels.
Consequently, $\M$ is E-stable.
\end{proof}

For a subcategory $\M$ of $\Mod A$, we denote by $\supp\M$ the set of prime ideals $\p$ of $A$ such that $\p\in\supp M$ for some $M\in\M$.
For a subcategory $\X$ of $\D(A)$, we denote by $\X_0$ the subcategory of $\Mod A$ consisting of all $A$-modules $M$ with $C_M\in\X$.
Now we can construct the following one-to-one correspondence.

\begin{thm}\label{mp}
One has maps
$$
\AC
\begin{smallmatrix}
f\\
\longrightarrow\\
\longleftarrow\\
g
\end{smallmatrix}
\AB
$$
defined by $f(\M)=\supp\M$ and $g(\Phi)=(\supp^{-1}\Phi)_0$.
The map $f$ is an inclusion-preserving bijection and $g$ is its inverse map.
\end{thm}

\begin{proof}
Let $\Phi$ be a subset of $\Spec A$, and put $\M=(\supp^{-1}\Phi)_0$.
We observe by Lemma \ref{supp}(2) that $\M$ is closed under direct sums and summands.
Fix an $A$-module $M$.
According to \cite[Lemma 3.3]{Krause} and Lemma \ref{ass}(2), we have
\begin{align*}
M\in\M & \ \Longleftrightarrow\ \supp M\subseteq\Phi \\
& \ \Longleftrightarrow\ \Ass E^i(M)\subseteq\Phi\text{ for all }i\ge 0 \\
& \ \Longleftrightarrow\ \supp E^i(M)\subseteq\Phi\text{ for all }i\ge 0 \\
& \ \Longleftrightarrow\ E^i(M)\in\M\text{ for all }i\ge 0.
\end{align*}
Hence $\M$ is E-stable, which says that the map $g$ is well-defined.

Let $\M$ be an E-stable subcategory of $\Mod A$ closed under direct sums and summands.
It is obvious that $\M$ is contained in $(\supp^{-1}(\supp\M))_0$.
Let $N$ be an $A$-module with $\supp N\subseteq\supp\M$.
Then we see from \cite[Lemma 3.3]{Krause} that for each $i\ge 0$ and $\p\in\Ass E^i(N)$ there exists a module $M\in\M$ and an integer $j\ge 0$ such that $\p\in\Ass E^j(M)$.
Hence $E(A/\p)$ is isomorphic to a direct summand of $E^j(M)$.
The module $E^j(M)$ is in $\M$ since $\M$ is E-stable, and $E(A/\p)$ is also in $\M$ since $\M$ is closed under direct summands.
Therefore by Lemma \ref{ass}(1) the module $E^i(N)$ is in $\M$ for every $i\ge 0$ since $\M$ is closed under direct sums, and $N$ is also in $\M$ since $\M$ is E-stable.
Thus we conclude that the composite map $gf$ is the identity map.

Let $\Phi$ be a subset of $\Spec A$.
It is obvious that $\supp((\supp^{-1}\Phi)_0)$ is contained in $\Phi$.
For $\p\in\Phi$ we have $\supp E(A/\p)=\Ass E(A/\p)=\{\p\}\subseteq\Phi$ by Lemma \ref{ass}(2).
This implies that $\Phi$ is contained in $\supp((\supp^{-1}\Phi)_0)$, and we conclude that the composite map $fg$ is the identity map.
\end{proof}

We say that a subcategory $\M$ of $\Mod A$ is {\em closed under short exact sequences} provided that for any short exact sequence $0\to L\to M\to N\to 0$ of $A$-modules, if two of $L$, $M$ and $N$ are in $\M$, then so is the third.
We say that $\M$ is {\em closed under left resolutions} provided that for any exact sequence $\cdots\to M_2\to M_1\to M_0\to N\to 0$ of $A$-modules, if every $M_i$ is in $\M$ then so is $N$.
Theorem \ref{mp} yields the following result.

\begin{cor}
Let $\M$ be an E-stable subcategory of $\Mod A$ closed under direct sums and summands.
Then $\M$ is closed under short exact sequences and left resolutions.
\end{cor}

\begin{proof}
By virtue of Theorem \ref{mp} there exists a subset $\Phi$ of $\Spec A$ such that $\M=(\supp^{-1}\Phi)_0$.

Let $0\to L\to M\to N\to 0$ be an exact sequence of $A$-modules.
Assume that two of $L$, $M$ and $N$, say $L$ and $M$, are in $\M$.
Then $\supp L$ and $\supp M$ are contained in $\Phi$, and so is $\supp N$ by Lemma \ref{supp}(1).
Hence $N$ is also in $\M$, and therefore $\M$ is closed under short exact sequences.

Let $\cdots\to M_2\to M_1\to M_0\to N\to 0$ be an exact sequence of $A$-modules with $M_i\in\M$ for any $i\ge 0$.
Then we have an $A$-complex $X=(\cdots\to X^{-2}\to X^{-1}\to X^0\to 0)$ with $X^{-i}=M_i$ for $i\ge 0$ which is quasi-isomorphic to $N$.
Lemma \ref{supp}(3) implies that $\supp N=\supp X\subseteq\bigcup_{i\le 0}\supp X^i\subseteq\Phi$.
Therefore $N$ is in $\M$.
\end{proof}

An $A$-complex $X$ is called {\em K-injective} if every morphism from an acyclic $A$-complex to $X$ is null-homotopic.
An $A$-complex $I$ is called a {\em minimal K-injective resolution} of an $A$-complex $X$ if there exists a quasi-isomorphism $X\to I$, each $I^i$ is an injective module, $I$ is a K-injective complex, and the kernel of the differential map $I^i\to I^{i+1}$ is an essential submodule of $I^i$ for all $i\in\Z$.
Every $A$-complex admits a minimal K-injective resolution; see \cite[before Proposition 5.1]{Krause}.

For a subcategory $\M$ of $\Mod A$, we denote by $\widetilde\M$ the localizing subcategory of $\D(A)$ generated by $\M$, and by $\overline\M$ the localizing subcategory of $\D(A)$ consisting of all complexes each of whose homology modules is in $\M$.
For an $A$-complex $X$ and an integer $i$, let $Z^i(X)$ (respectively, $B^i(X)$) denote the $i$th cycle (respectively, boundary) of $X$.

\begin{prop}\label{aboutmp}
\begin{enumerate}[\rm (1)]
\item
Let $\M$ be an E-stable subcategory of $\Mod A$ closed under direct sums and summands.
Then
$$
\supp^{-1}(\supp\M)=\widetilde\M.
$$
\item
Let $\M$ be a thick subcategory of $\Mod A$ closed under direct sums.
Then
$$
\supp^{-1}(\supp\M)=\overline\M.
$$
\end{enumerate}
\end{prop}

\begin{proof}
(1) Set $\X=\supp^{-1}(\supp\M)$.
We see from Proposition \ref{-1} that $\X$ is a localizing subcategory of $\D(A)$ containing $\M$.
Hence $\X$ contains $\widetilde\M$.
Corollary \ref{gene}(1) says that $\X$ is the localizing subcategory of $\D(A)$ generated by $\{ k(\p)\}_{\p\in\supp\M}$.
Hence we have only to show that $k(\p)$ belongs to $\widetilde\M$ for every $\p\in\supp\M$.

Fix a prime ideal $\p$ in $\supp\M$.
Then $k(\p)\otimes_A^{\bf L}M$ is nonzero for some $M\in\M$.
Since $M$ is in $\widetilde\M$, the complex $k(\p)\otimes_A^{\bf L}M$ is in $\widetilde\M$ by \cite[(2.1.7)]{Neeman}.
Note that $k(\p)\otimes_A^{\bf L}M$ is isomorphic to a nonzero direct sum of shifts of $k(\p)$.
Since $\widetilde\M$ is closed under shifts and direct summands by Corollary \ref{gene}(2), $k(\p)$ is in $\widetilde\M$, as required.

(2) Fix an $A$-complex $X$.
We want to prove that $\supp X\subseteq\supp\M$ if and only if $H^i(X)\in\M$ for all integers $i$.

Suppose that the inclusion relation $\supp X\subseteq\supp\M$ holds.
Let $I$ be a minimal K-injective resolution of $X$.
Then we see from \cite[Lemma 3.3 and Proposition 5.1]{Krause} that for each $i\in\Z$ and $\p\in\Ass I^i$ there exists a module $M\in\M$ and an integer $j\ge 0$ such that $\p\in\Ass E^j(M)$.
Hence $E(A/\p)$ is isomorphic to a direct summand of $E^j(M)$.
We observe from \cite[Lemma 3.5]{Krause} and the closedness of $\M$ under cokernels that $E^j(M)$ is in $\M$, and from the closedness of $\M$ under direct summands that $E(A/\p)$ is also in $\M$.
Therefore each $I^i$ is in $\M$ by Lemma \ref{ass}(1) as $\M$ is closed under direct sums.
For every $i\in\Z$ there are exact sequences of $A$-modules:
\begin{align*}
& 0 \to Z^i(I) \to I^i \to I^{i+1},\\
& 0 \to Z^i(I) \to I^i \to B^{i+1}(I) \to 0,\\
& 0 \to B^i(I) \to Z^i(I) \to H^i(X) \to 0.
\end{align*}
Since $\M$ is closed under kernels and cokernels, from these exact sequences we observe that $H^i(X)$ is in $\M$ for every $i\in\Z$.

Conversely, suppose that all $H^i(X)$ belong to $\M$.
Then $\supp H^i(X)$ is contained in $\supp\M$ for all $i\in\Z$.
Here note from \cite[Theorem 3.1]{Krause} that $\supp\M$ is a coherent subset of $\Spec A$.
Therefore it follows by \cite[Theorem 5.2]{Krause} that $\supp X$ is contained in $\supp\M$, as desired.
\end{proof}

Now we are in a position to prove our main theorem which we stated in Introduction.

\begin{mpf}
The first commutative diagram of bijections in Main Theorem is obtained from Theorems \ref{neeman}(1), \ref{mp} and Proposition \ref{aboutmp}(1).
Theorem \ref{mt}, \cite[Theorem 3.1]{Krause} and Proposition \ref{aboutmp}(2) make the second diagram in Main Theorem.
Theorem \ref{neeman}(2), \cite[Corollary 3.6]{Krause} and Proposition \ref{aboutmp}(2) give the third one.
All these three commutative diagrams together with Proposition \ref{thest}, Corollary \ref{lchst}(2) and \cite[Proposition 4.1(2)]{Krause} yield the last diagram in Main Theorem.
\qed
\end{mpf}

\section*{Acknowledgments}

The author is indebted to the referee for his/her valuable comments.
The author also thanks Hiroki Miyahara very much for his helpful comments.



\begin{thebibliography}{99}

\bibitem{Balmer}
{\sc P. Balmer},
The spectrum of prime ideals in tensor triangulated categories.
{\it J. Reine Angew. Math.} {\bf 588} (2005), 149--168.

\bibitem{BH}
{\sc W. Bruns}; {\sc J. Herzog},
Cohen-Macaulay rings. revised edition.
Cambridge Studies in Advanced Mathematics, 39,
{\it Cambridge University Press, Cambridge}, 1998.

\bibitem{BIK}
{\sc D. Benson}; {\sc S. Iyengar}; {\sc H. Krause},
Local cohomology and support for triangulated categories.
{\it Ann. Sci. \'{E}cole Norm. Sup. (4)}, to appear.

\bibitem{Gabriel}
{\sc P. Gabriel},
Des cat\'{e}gories ab\'{e}liennes.
{\it Bull. Soc. Math. France} {\bf 90} (1962), 323--448.

\bibitem{GP}
{\sc G. Garkusha}; {\sc M. Prest},
Classifying Serrre subcategories of finitely presented modules.
{\it Proc. Amer. Math. Soc.} {\bf 136} (2008), no. 3, 761--770.

\bibitem{GP2}
{\sc G. Garkusha}; {\sc M. Prest},
Torsion classes of finite type and spectra, pp. 393-412 in {\it K-theory and Noncommutative Geometry}, European Math. Soc., 2008.

\bibitem{Hopkins}
{\sc M. J. Hopkins},
Global methods in homotopy theory.
{\it Homotopy theory (Durham, 1985)}, 73--96, London Math. Soc. Lecture Note Ser., 117, {\it Cambridge Univ. Press, Cambridge}, 1987.

\bibitem{Hovey}
{\sc M. Hovey},
Classifying subcategories of modules.
{\it Trans. Amer. Math. Soc.} {\bf 353} (2001), no. 8, 3181--3191.

\bibitem{Krause}
{\sc H. Krause},
Thick subcategories of modules over commutative Noetherian rings (with an appendix by Srikanth Iyengar).
{\it Math. Ann.} {\bf 340} (2008), no. 4, 733--747.

\bibitem{Neeman}
{\sc A. Neeman},
The chromatic tower for $D(R)$.
With an appendix by Marcel B\"{o}kstedt.
{\it Topology} {\bf 31} (1992), no. 3, 519--532.

\bibitem{Neemanbook}
{\sc A. Neeman},
Triangulated categories. 
Annals of Mathematics Studies, 148. {\it Princeton University Press, Princeton, NJ}, 2001.

\bibitem{wide}
{\sc R. Takahashi},
Classifying subcategories of modules over a commutative noetherian ring.
{\it J. Lond. Math. Soc. (2)} {\bf 78} (2008), no. 3, 767--782.

\bibitem{Thomason}
{\sc R. W. Thomason},
The classification of triangulated subcategories.
{\it Compositio Math.} {\bf 105} (1997), no. 1, 1--27.


\end{thebibliography}
\end{document}